\documentclass[11pt]{amsart}
\usepackage{amscd}
\usepackage{graphicx}
\newtheorem{theorem}{Theorem}[section]
\newtheorem{proposition}[theorem]{Proposition}
\newtheorem{corollary}[theorem]{Corollary}
\newtheorem{lemma}[theorem]{Lemma}
\newtheorem{remark}[theorem]{Remark}

\newtheorem{example}[theorem]{Example}

\def\ml{\mathcal{C}}
\def\ml1{\mathcal{C}^1}
\def\mlb1{\mathcal{C}_{b}^{1}}

\def\g{\gamma}
\def\mc{\mathbb{C}}

\def\b{\beta}
\def\s{\sigma}
\def\a{\alpha}

\def\l{\lambda}

\def\frk{\frak}               

\def\Phi{{\frk n}}

\begin{document}
\textwidth=13cm

\title{Braid groups in complex projective spaces}

\thanks{
  This research is partially supported by Higher Education Commission, Pakistan.\\
2010 AMS Classification Primary: 20F36, 52C35, 57M05; Secondary: 51A20}
\author{ BARBU BERCEANU$^{1,2}$ ,\,\,SAIMA PARVEEN $^{2}$}
\address{$^{1}$Abdus Salam School of Mathematical Sciences,
 GC University, Lahore-Pakistan, and
 Institute of Mathematics Simion Stoilow, Bucharest-Romania\\(Permanent address).}
\email {Barbu.Berceanu@imar.ro}
\address{$^{2}$Abdus Salam School of Mathematical Sciences,
 GC University, Lahore-Pakistan.}
\email {saimashaa@gmail.com}
\maketitle
 \pagestyle{myheadings} \markboth{\centerline {\scriptsize
BARBU BERCEANU,\,\,\,SAIMA PARVEEN   }} {\centerline {\scriptsize
Braid groups in complex projective spaces }}
\begin{abstract} We describe the fundamental groups of ordered and unordered $k-$point sets in $\mc P^n$ generating a projective subspace of dimension $i$. We apply these to study connectivity of more complicated configurations of points.
\end{abstract}

\maketitle
\section{INTRODUCTION}
Let $M$ be a manifold and $\mathcal{F}_k(M)$ be {\em the ordered
configuration space} of $k$ distinct points
$\{(x_1,\ldots,x_k)\in M^k|x_i\neq x_j,\,\,i\neq j\}$. There is a
proper right action of $\Sigma_k$, the symmetric group of order $k$, on $\mathcal{F}_k(M)$.
The orbit space $\mathcal{F}_{k}(M)/\Sigma _k$ is the {\em unordered configuration space}, denoted
$\mathcal{C}_k(M)$, and the natural projection $\mathcal{F}_k(M)\longrightarrow \mathcal{C}_k(M)$ is a regular
covering. For a simply connected manifold $M$ of dimension $\geq 3$, the {\em pure braid group} $\pi_1(\mathcal{F}_k(M))$ is trivial and the {\em braid group}  $\pi_1(\mathcal{C}_k(M))$ is isomorphic to $\Sigma_k$; as an example, $\pi_1(\mathcal{F}_{k}(\mc P^k))=1$ for $k\geq 2$.
In low dimensions there are non trivial pure braids. The pure braid
group of the plane, denoted by $\mathcal{PB}_n$, has the presentation
\cite{F}
$$\pi_1(\mathcal{F}_n(\mc))=\mathcal{PB}_n\cong\big<\a_{ij}\,,\,\,1\leq i<j\leq
n\big|(YB\,3)_n,(YB\,4)_n\big>$$ where the Yang-Baxter relations
$(YB\,3)_n$ and $(YB\,4)_n$ are, for any $1\leq i <j< k\leq n$,
 \begin{equation*}
(YB\,3)_n: \a_{ij}\a_{ik}\a_{jk}=\a_{ik}\a_{jk}\a_{ij}=\a_{jk}\a_{ij}\a_{ik}
\end{equation*}
and, for any $1\leq i<j<k<l\leq n$,
$$
(YB\,4)_n: [\a_{kl},\a_{ij}]=[\a_{il},\a_{jk}]=[\a_{jl},\a_{jk}^{-1}\a_{ik}\a_{jk}]=[\a_{jl},\a_{kl}\a_{ik}\a_{kl}^{-1}]=1.
$$
 The  braid group of the plane, denoted by $\mathcal{B}_n$, has the classical Artin  presentation \cite{A} $$\pi_1(\mathcal{C}_n(\mc))=\mathcal{B}_n\cong\big<\s_i,\,i=1,\dots,{n-1}|(A)_n\big>$$ where the Artin relations $(A)_n$ are
$$(A)_n:\left\{
  \begin{array}{ll}
   &\s_i\s_j=\s_j\s_i,\,\,\,\,\,\,\forall\,\,\,i,j=1,\ldots,n-1\,\,\,\hbox{with}\,\,|i-j|\geq 2,\\
  &\s_i\s_{i+1}\s_i=\s_{i+1}\s_i\s_{i+1}\,\,\,\,\hbox{for}\,\,\,i=1,\ldots,n-2.
  \end{array}
\right.
$$
The pure braid and the braid groups of $S^2\approx\mc P^1$ have the
presentations \cite{B2},\cite{F}
$$\pi_1(\mathcal{F}_{k+1}(S^2))\cong\big<\a_{ij}, 1\leq i<j\leq k\,\big|(YB\,3)_k,(YB\,4)_k, D_{k}^2=1\big>$$
$$\pi_1(\mathcal{C}_{k+1}(S^2))\cong\big<\s_i,1\leq i\leq k\,\big|(A)_{k+1},\,\,\s_1\s_2\ldots\s_k^2\ldots\s_2\s_1=1\big>,$$
where $D_k=\a_{12}\a_{13}\a_{23}\ldots\a_{1k}\ldots\a_{k-1,k}$.

The inclusion morphisms $\mathcal{PB}_n\rightarrow\mathcal{B}_n$ for $\mc$ and $\mc P^1$ are given by (see \cite{B2})
$$\a_{ij}\mapsto\s_{j-1}\s_{j-2}\ldots\s_{i+1}\s_{i}^2\s_{i+1}^{-1}\ldots\s_{j-1}^{-1}\,\,$$
(due to these inclusions, we can identify the pure braid $D_k$ with $\Delta_k^2$, the square of Garside braid).

Using the geometrical structure of projective spaces we stratify the
configuration spaces $\mathcal{F}_k(\mc P^n)$ and
$\mathcal{C}_k(\mc P^n)$ with complex submanifolds as follows:
$$\mathcal{F}_k(\mathbb{C}P^n)=\mathop{\coprod}\limits_{i=1}^{n} \mathcal{F}_{k}^{i,n}\,\,,$$
where
$ \mathcal{F}_{k}^{i,n}$ is the ordered configuration space of all $k$ points in $\mc P^n$
generating a subspace  of dimension $i$, and
$$\mathcal{C}_k(\mathbb{C}P^n)=\mathop{\coprod}\limits_{i=1}^{n} \mathcal{C}_{k}^{i,n}\,\,,$$
where
$ \mathcal{C}_{k}^{i,n}$ is the unordered configuration space of all $k$ points in $\mc P^n$
generating a subspace of dimension $i$. Obviously, $\mathcal{C}_{k}^{i,n}\cong \mathcal{F}_{k}^{i,n}/\Sigma_k$.
\begin{theorem}\label{th:1}
The spaces $\mathcal{F}_{k}^{i,n}$ are simply connected with the following exceptions
\begin{enumerate}
  \item for $k\geq2$,
  $$\pi_1(\mathcal{F}_{k+1}^{1,1})\cong\big<\a_{ij},\,\,\,1\leq i<j\leq k\,\big|(YB\,3)_k,(YB\,4)_k,D_{k}^2=1\big>;$$
  \item for $k\geq 3$ and $ n\geq 2$,
  $$\pi_1(\mathcal{F}_{k+1}^{1,n})\cong\big<\a_{ij},\,\,\,1\leq i<j\leq k\,\big|(YB\,3)_k,(YB\,4)_k,\,D_{k}=1\big>.$$
\end{enumerate}
\end{theorem}
In this list of non simply connected spaces, only
$\mathcal{F}_{3}^{1,1}$ has finite fundamental group and this is isomorphic to $
\mathbb{Z}_2$.
\begin{corollary}\label{l.1.2}
The first homology groups $H_1(\mathcal{F}_{k}^{i,n})$ are trivial with the following exceptions
\begin{enumerate}
  \item for $k\geq 2$, \,\,\,$H_1(\mathcal{F}_{k+1}^{1,1})=\mathbb{Z}^{{k\choose 2}-1}\oplus\mathbb{Z}_2$;
  \item for $k\geq3 \,\hbox{and}\,\,n\geq2,\,\,\,H_1(\mathcal{F}_{k+1}^{1,n})=\mathbb{Z}^{{k\choose 2}-1}$.
\end{enumerate}
\end{corollary}
\begin{theorem}\label{th:2}
The fundamental group of $\mathcal{C}_{k}^{i,n}$ is isomorphic to
$\Sigma_k$ with the following exceptions
\begin{enumerate}
  \item for $k\geq 2$,
  $$\pi_1(\mathcal{C}_{k+1}^{1,1})\cong\big<\s_1\ldots\s_{k}|(A )_k,\,\s_1\ldots\s_{k}^2\ldots\s_1=1\big>;$$
\item for $k\geq 3$ and $n\geq 2$, $$\pi_1(\mathcal{C}_{k+1}^{1,n})\cong\big<\s_1\ldots\s_{k}|(A)_k,\,\s_1\ldots\s_{k}^2\ldots\s_1=1,\Delta_{k}^2=1\big>,$$
where $\Delta_k=\s_1\s_2\s_1\s_3\s_2\s_1\ldots\s_{k-1}\ldots \s_1.$
\end{enumerate}
\end{theorem}
The space $\mathcal{C}_{k}^{i,n}$ has a finite fundamental group only for $2\leq i\leq \min(k+1,n)$,  and in this case $\pi_1(\mathcal{C}_{k}^{i,n})\cong\Sigma_k$, and also $\pi_1(\mathcal{C}_{3}^{1,1})$ is the  dicyclic group of order 12.
\begin{corollary}\label{l.1.4} The homology groups $H_1(\mathcal{C}_{k+1}^{i,n})$ are isomorphic to $\mathbb{Z}_2$ with the following exceptions
\begin{enumerate}
  \item for $k\geq 2$, $H_1(\mathcal{C}_{k+1}^{1,1})=\mathbb{Z}_{2k}$ ;
  \item for $k\geq 3$ and $n\geq 2$, $H_1(\mathcal{C}_{k+1}^{1,n})= \left\{
                                         \begin{array}{ll}
                                          \mathbb{Z}_{k}\,\,\,\hbox{for}\,\,k=even, \\
                                           \mathbb{Z}_{2k}\,\,\,\hbox{for}\,\,k=odd\, .
                                         \end{array}
                                       \right.$
\end{enumerate}
\end{corollary}
V.L. Moulton studied in \cite{M2} a related problem: braids for $m$ points in general position in $\mc P^n$ (any subset of $n+1$ points spans $\mc P^n$) and the analogous affine problem. The only intersection of this paper with \cite{M2} is the simply connectedness of $\mathcal{F}_{n+1}^{n,n}$ ($Y_{n+1}^n$ in Moulton notations).

Section $2$ contains the geometrical part of the paper:
local triviality of some natural fibrations associated to
$\mathcal{F}_{k}^{i,n}$ and $\mathcal{C}_{k}^{i,n}$, and also in
this section we meet the pure braid $\Delta_{k}^2 $ ($\Delta_k$ is
the fundamental Garside \cite{G} braid in $\mathcal{B}_k$) playing the main
role in homotopical computations (see Lemma \ref{l.6}).

The proofs of Theorems \ref{th:1} and \ref{th:2} and their
Corollaries are given in Section $3$.

In the last section we compute the fundamental groups of $\mathcal{P}$, Pappus' configuration space.
\begin{theorem}\label{th:3}
The fundamental group of $\mathcal{P}$ is isomorphic to $\mathbb{F}_2\times\mathbb{F}_2$.
\end{theorem}
\noindent($\mathbb{F}_n$ is the free group with $n$ generators).

In the Appendix one can find some words representing $\Delta_n$ and
$D_n=\Delta_{n}^2$.
\section{Local triviality of geometric fibrations}
We begin with some simple remarks on the stratification with complex
submanifolds
$$\mathcal{F}_k(\mathbb{C}P^n)=\mathop{\coprod}\limits_{i=1}^{n} \mathcal{F}_{k}^{i,n}.$$

\begin{remark}
\begin{enumerate}
  \item ${\mathcal{F}_{k}^{i,n}}\neq \emptyset$ if and only if $i\leq \min(k+1,n)$;
\item  ${\mathcal{F}_{k}^{1,1}}=\mathcal{F}_k(\mc P^1)$,\,\,\,${\mathcal{F}_{2}^{1,n}}=\mathcal{F}_2(\mc P^n)$;
  \item the adjacency of the strata is given by $$\overline{\mathcal{F}_{k}^{i,n}}=\mathcal{F}_{k}^{1,n}\coprod\ldots\coprod\mathcal{F}_{k}^{i,n}.$$
\end{enumerate}
\end{remark}

Grasmannian manifolds are related to the spaces $\mathcal{F}_{k}^{i,n}$ through the following fibrations:
\begin{proposition}\label{pr:1}
The projection
$$\gamma:\mathcal{F}_{k}^{i,n}\rightarrow  Gr^i(\mc P^n)$$
given by
$$(x_1,\ldots,x_k)\mapsto i-th\,\,\hbox {dimensional space generated by \,\,}\{x_1,x_2,\ldots,x_k\}$$
is a locally trivial fibration with fiber $\mathcal{F}_{k}^{i,i}$.
\end{proposition}
\noindent\textit{Proof}. Take $V_0\in Gr^i(\mc P^n)$ and choose $L_{0}\in Gr^{n-i-1}(\mc P^n)$ such that $L_0\cap V_0=\emptyset$  and
 let define $\mathcal{U}_{L_0}$, an open neighborhood of $V_0$, by
$$\mathcal{U}_{L_0}=\{V\in Gr^i(\mc P^n)|L_0\cap V=\emptyset\}.$$
Take $k$ distinct points  $(x_1,\ldots,x_k)$ in $V_0$ such that
$\hbox{span}(x_1,\ldots,x_k)=V_0$. For any arbitrary $i$-plane $V$ in $\mathcal{U}_{L_0}$, define the projective isomorphism
$$ \varphi_{_{V}}:V_0\rightarrow V,\,\,\varphi_{_{V}}(x)=(x\vee L_0)\cap V.$$
The local trivialization is given by the homeomorphism
 $$f:\gamma^{-1}(\mathcal{U}_{L_0})\rightarrow \mathcal{U}_{L_0}\times \mathcal{F}_k(V_0)$$
 $$y=(y_1,\ldots,y_k)\mapsto\big(\gamma(y_1,\ldots,y_k),(\varphi_{_{\gamma(y)}}^{-1}(y_1),\ldots,\varphi_{_{\gamma(y)}}^{-1}(y_k))\big)$$
(where $\mathcal{F}_k(V_0)\approx\mathcal{F}_{k}^{i,i}$), making the diagram commutative

\begin{center}
\begin{picture}(360,120)
\thicklines
\put(77,90){\vector(2,-1){45}}
\put(107,100){\vector(1,0){70}}
\put(197,90){\vector(-2,-1){45}}
\put(55,98){${\gamma}^{-1}(\mathcal{U}_{L_0})$}
\put(184,98){$\mathcal{U}_{L_0}\times \mathcal{F}_{k}^{i,i}$}
\put(130,60){$ \mathcal{U}_{L_0}$}
\put(83,70){$\gamma$}
\put(180,70){$pr_1$}
\put(134,105){$f$}
\put(350,60){$\square$}
\end{picture}
\end{center}
\vspace{-1.3cm}
\begin{corollary}
The complex dimensions of the strata are given by
\end{corollary}
\vspace{-0.7cm}
$$\hbox{dim}(\mathcal{F}_{k}^{i,n})=\hbox{dim}(\mathcal{F}_{k}^{i,i})+\hbox{dim}( Gr^{i}(\mc P^n))=ki+(i+1)(n-i).$$
\begin{proof}
$\mathcal{F}_{k}^{i,i}$ is a Zariski open subset in $(\mc P^i)^k$ for $k\geq i+1$.
\end{proof}
\begin{lemma}\label{l.7}
The projection
$$p:\mathcal{F}_{k+1}^{k,k}\longrightarrow \mathcal{F}_{k}^{k-1,k}
,\,\,\,\,(x_1,\ldots,x_{k+1})\mapsto(x_1,\ldots,x_{k})$$
is a locally trivial fibration with fiber
$\mc P^k\setminus \mc P^{k-1}\approx\mc^k$.
\end{lemma}
\begin{proof}
Take $(x_{1}^0,\ldots,x_{k+1}^0)\in \mathcal{F}_{k+1}^{k,k}$ and fix $x_{k+2}^0$ such that
any $k+1$ points of the set $\{x_{1}^0,\ldots,x_{k+1}^0,x_{k+2}^0\}$ are independent. Define the open neighborhood $\mathcal{U}$ of $(x_{1}^0,\ldots,x_{k}^0)$ by
$$\mathcal{U}=\{(x_1,\ldots,x_k)|\,\hbox{the set}\, \{x_1,\ldots,x_k,x_{k+1}^0,x_{k+2}^0\} \,\hbox{is in general position} \}.$$
There exists a unique projective isomorphism $T_{(x_1,\ldots,x_k)}:\mc P^k\longrightarrow\mc P^k,$
which depends continuously on $ (x_1,\ldots,x_k) $, such that
$$T_{(x_1,\ldots,x_k)}(x_{1}^0,\ldots,x_{k}^0,x_{k+1}^0,x_{k+2}^0)=(x_1,\ldots,x_k,x_{k+1}^0,x_{k+2}^0)$$
(see \cite{B1}). We can define the homeomorphisms $\varphi,\psi$ by :
 $$p^{-1}(\mathcal{U})\mathop{\longleftarrow}\limits_{\psi}^{\mathop{\longrightarrow}\limits^{\varphi}} \mathcal{U}\times \big(\mc P^k\setminus\hbox{span}\{x_1^0,\ldots,X_k^0\}\big)$$
 $$\varphi(x_1,\ldots,x_k,x)=((x_1,\ldots,x_k),T_{(x_1,\ldots,x_k)}^{-1}(x))$$
 $$\psi((x_1,\ldots,x_k),y)=(x_1,\ldots,x_k,T_{(x_1,\ldots,x_k)}(y))$$
 satisfying $pr_1\circ\varphi=p$ .
\end{proof}
\begin{remark}
More generally, the projection
$$p:\mathcal{F}_{k+1}^{k,n}\longrightarrow \mathcal{F}_{k}^{k-1,n},\,(x_1,\ldots,x_{k+1})\mapsto(x_1,\ldots,x_{k})$$
is a locally trivial fibration with simply connected fiber
$\mc P^n\setminus \mc P^{k-1}$.
\end{remark}
\begin{proof}
The proof is similar to the previous one: take $(x_{1}^0,\ldots,x_{k+1}^0)\in \mathcal{F}_{k}^{k,n}$ and fix $x_{k+2}^0,..,x_{n+2}^0$ such that
any $n+1$ points of the set $\{x_{1}^0,..,x_{k+1}^0,..,x_{n+2}^0\}$ are independent. Define the open neighborhood $\mathcal{U}$ of $(x_{1}^0,\ldots,x_{k}^0)$ by
$$\mathcal{U}=\{(x_1,\ldots,x_k)|\,\hbox{the set}\, \{x_1,..,x_k,x_{k+1}^0,..,x_{n+2}^0\} \,\hbox{is in general position} \}$$
and construct the trivialization as in Lemma \ref{l.7}. The fiber
$\mc P^n\setminus \mc P^{k-1}$ is simply connected because the real
codimension of $\mc P^{k-1}$ is $\geq 4$ (for $k\leq n-1$) and it is contractible for $k=n$.
\end{proof}
Let $\mathcal{A}=(A_1,\ldots,A_p)$ be a sequence of subsets of
$\{1,\ldots,k\}$ and the integers $d_1,\ldots,d_p$ given by
$d_j=|A_j|-1,\,\,\,j=1,\ldots,p$. Let us define
$$\mathcal{F}_{k}^{\mathcal{A},n}=\{(x_1,\ldots,x_k)\in\mathcal{F}_{k}(\mc P^n)\big|\dim<x_i>_{i\in A_j}=d_j \}.$$
 \begin{example}\label{ex:2.6}
 \begin{enumerate}
    \item If $\mathcal{A}=\{A_1\},\,\,A_1=\{1,\ldots,k\}$, then $\mathcal{F}_{k}^{\mathcal{A},n}=\mathcal{F}_{k}^{k-1,n}.$
    \item if all $A_i$ have cardinality $|A_i|\leq 2$, then $\mathcal{F}_{k}^{\mathcal{A},n}=\mathcal{F}_k(\mc P^n).$
   \item if $p\geq2$ and $|A_p|\leq 2$, then $\mathcal{F}_k^{(A_1,\ldots,A_p),n}=\mathcal{F}_k^{(A_1,\ldots,A_{p-1}),n}$.
   \item if $p\geq2$ and $A_p\subseteq A_1$, then $\mathcal{F}_k^{(A_1,\ldots,A_p),n}=\mathcal{F}_k^{(A_1,\ldots,A_{p-1}),n}$.
    \item $\mathcal{F}_{k}^{i,n}=\bigcup\big\{\mathcal{F}_{k}^{\mathcal{A},n}\big|\mathcal{A}=\{A\},A\in
{\{1,\ldots,k\} \choose{i+1}}\big\}.$
\item $Y_{k}^{n}=\bigcap\big\{\mathcal{F}_{k}^{\mathcal{A},n}\big|\mathcal{A}=\{A\},A\in
{\{1,\ldots,k\} \choose{n+1}}\big\}$ is the space analyzed by Moulton in \cite{M2}.
 \end{enumerate}
\end{example}
\begin{lemma}\label{l.6.5}
For $A=\{1,\ldots,i+1\}$ and $k>i+1$ the map
$$P_A:\mathcal{F}_{k}^{\{A\},i}\rightarrow\mathcal{F}_{i+1}^{i,i},\,\,\,(x_1,\dots,x_k)\mapsto(x_1,\ldots,x_{i+1})$$
is a locally trivial fibration with fiber $\mathcal{F}_{k-i-1}(\mc
P^i\setminus\{x_1,\ldots,x_{i+1}\}).$
\end{lemma}
\begin{proof}
Choose $z_{i+2}\in \mc P^i$ in general position with the set $ (x_1,\ldots,x_{i+1}) $ and define the neighborhood
$\mathcal{U}$ of $(x_1,\ldots,x_{i+1})\in \mathcal{F}_{i+1}^{i,i}$ by $\mathcal{U}=\{(y_1,\ldots,y_{i+1})\in \mathcal{F}_{i+1}^{i,i}| \,(y_1,\ldots,y_{i+1},z_{i+2})\hbox{ are
in general position}\}$. There exists a unique projective isomorphism $F_y:\mc P^i\rightarrow \mc P^i$,
which depends continuously on
$y=(y_{1},\ldots,y_{i+1})$, such that
$$F_y(x_1,\ldots,x_{i+1},z_{i+2})=(y_1,\ldots,y_{i+1},z_{i+2})$$
and this gives a local trivialization
 $$f:P_{A}^{-1}(\mathcal{U})\rightarrow \mathcal{U}\times \mathcal{F}_{k-i-1}(\mc
P^i\setminus\{x_1,\ldots,x_{i+1}\})$$
$$(y_1,\ldots,y_k)\mapsto\big((y_1,\ldots,y_{i+1}),F_{y}^{-1}(y_{i+2},\ldots,y_k)\big)$$
%
which satisfies $pr_1\circ f=P_A$ .
\end{proof}
\begin{remark}
If $k=i+1$, then $P_A$ is the identity map.
\end{remark}
 Given $\mathcal{A}=(A_1,\ldots,A_p),\,d=(d_1,\ldots,d_p)$ as before and an index $h\in\{1,\ldots,k\}$, we define $\mathcal{A}'=(A'_1,\ldots,A'_p),\,d=(d'_1,\ldots,d'_p)$ by:
$$A'_j=\left\{
   \begin{array}{ll}
      A_j,\,\,\hbox{if}\,\,h\notin A_j\\
    A_j\setminus\{h\},\,\,\hbox{if}\,\,h\in A_j
   \end{array}
 \right.
,\,\,\,\,d'_j=\left\{
   \begin{array}{ll}
       d_j,\,\,\hbox{if}\,\,h\notin A_j\\
      d_j-1,\,\,\hbox{if}\,\,h\in A_j
   \end{array}
 \right..$$

\begin{lemma}\label{l.2.9}
The map
$$p_h:\mathcal{F}_{k}^{\mathcal{A},n}\rightarrow\mathcal{F}_{k-1}^{\mathcal{A}',n},\,\,
(x_1,\ldots,x_k)\mapsto(x_1,\ldots,\widehat{x_h},\ldots,x_k)$$
has local sections with path-connected fibers.
\end{lemma}
\begin{proof}
Let us suppose that $h=k$ and $k\in(A_1\cap\ldots\cap A_l)\setminus( A_{l+1}\cup\ldots \cup A_p)$. Then the fiber of the map
$pr_k:\mathcal{F}_{k}^{\mathcal{A},n}\rightarrow\mathcal{F}_{k-1}^{\mathcal{A}',n}$
is $$pr_{k}^{-1}(x_1,\ldots,x_{k-1})\approx \mc P^n\setminus
\big(L'_{1}\cup\ldots\cup L'_{l}\cup\{x_1,\ldots,x_{k-1}\}\big)$$
where $L'_{j}=\hbox{span}(x_i)_{i\in A'_j}$. Even in the case when $\dim
L_j=n$, we have $\dim L'_j<n$, hence the fiber is path-connected and
nonempty. Fix a base point $x=(x_1,\ldots,x_{k-1})\in \mathcal{F}_{k-1}^{\mathcal{A}',n}$ and choose $x_k\in\mc P^n\setminus(L'_1\cup\ldots \cup L'_p\cup\{x_1,\ldots,x_{k-1}\})$. There are neighborhoods $W_j$ of $L'_j\,(j=1,\ldots,l)$  such that
$x_k\notin L''_j$ if $L''_j\in W_j$; we take a constant local
section
$$s:W=g^{-1}\big((\mc P^n\setminus\{x_k\})^{k-1}\times\mathop{\prod}\limits_{i=1}^{l}W_i\big)\rightarrow \mathcal{F}_{k}^{\mathcal{A},n}
$$
$$(y_1,\ldots,y_{k-1})\mapsto(y_1,\ldots,y_{k-1},x_k),$$
where the continuous map $g$ is given by:
$$g:\mathcal{F}_{k-1}^{\mathcal{A}',n}\rightarrow(\mc P^n)^{k-1}\times Gr^{d'_1}(\mc P^{n})\times\ldots\times Gr^{d'_l}(\mc P^{n})$$
$$(y_1,\ldots,y_{k-1})\mapsto(y_1,\ldots,y_{k-1},L''_{1},\ldots,L''_p),$$
and $L''_j=\hbox{span}\{y_i\}_{i\in A'_{j}}$ for $j=1,\ldots,l$.
\end{proof}
\begin{lemma}\label{l.2.10}

Let $p:B\rightarrow C$ be a continuous map with local sections such that $C$ is path-connected and
$p^{-1}(y)$ are path-connected for all $y\in C$. Therefore $B$ is path-connected.
\end{lemma}
\begin{proof}
For any $b\in B$ there exists a neighborhood $W$ of $p(b)$ and a
section $s:W\rightarrow B$ such that $s(p(b))=b_0$. We can join $b$ with any point $b'$ in $V=p^{-1}(W)$ by a continuous path: $b$ to $b_0$ and $b'$ to
$b'_0=s(p(b'))$ in fibers and $b_0$ to $b'_0$ using a path in $ C $ and the section $s$.
For any $b_1,b_2\in B$, a path in $C$ from $p(b_1)$ to $p(b_2)$ can
be covered by finite open sets of the type described before and next we apply the previous construction.
\end{proof}
\begin{proposition}\label{l.5}
The space $\mathcal{F}_{k}^{\mathcal{A},n}$ is path-connected.
\end{proposition}
\begin{proof}
  Use the previous Lemma and induction on $p$ and $d_1+d_2+\ldots+d_p$. If $p=1$, use Lemma \ref{l.6.5} and the space $\mathcal{F}_{i+1}^{i,i}$ which is path-connected. If $A_p$ is not included in $A_1$ and $d_p\geq3$, delete a point in $A_p\setminus A_1$ and use Lemma \ref{l.2.9} and \ref{l.2.10}. If $A_p\subset A_1$ or $d_p\leq2$, use Example \ref{ex:2.6}, $(3)$ and $(4)$.
  \end{proof}

\begin{lemma}\label{l.6}
The homotopy class of the map
$$\g:S^1\longrightarrow \mathcal{F}_{k+1}(\mc P^1),\,\,\g(z)=([z:1],[2z:1],\ldots,[kz:1],[1:0])$$
corresponds to the following pure braid in $\pi_1(\mathcal{F}_{k+1}(\mc P^1))$:
$$[\g]=\a_{12}(\a_{13}\a_{23})\ldots(\a_{1k}\a_{2k}\ldots\a_{k-1,k})=D_{k}\,\,.$$
\end{lemma}
\begin{proof}
The loop $\g_{k+1}$ in $\mathcal{F}_{k+1}(\mc P^1)$ given by $$z\longmapsto([z:1],[2z:1],\ldots,[kz:1],[1:0])$$
fixes the point at $\infty$ and gives the pure braid in $\mc$
$$t\longmapsto(e^{2\pi it},2e^{2\pi it},\ldots,ke^{2\pi it})$$

By induction we assume that the class $[\g_k]$ in $\mathcal{F}_{k}(\mc P^1)$ is given by the element
$(\s_1\ldots\s_{k-2})\ldots(\s_1\s_2)\s_{1}^2(\s_2\s_1)\ldots(\s_{k-2}\ldots\s_1)$, where the $k-$points in $\mc P^1$ are $2,3,\ldots,k$ and $\infty$. In $\mathcal{F}_{k+1}(\mc P^1)$ we introduce a new strand corresponding to the point $1$. The image of the $j-$strand $t\rightarrow[je^{2\pi it}:1]$ lies on the cylinder $|z|=j,\,t\in[0,1]$, therefore the image of the strand corresponding to $1$ is interior to all the other cylinders and can be deformed to a straight line segment, like in the next figure

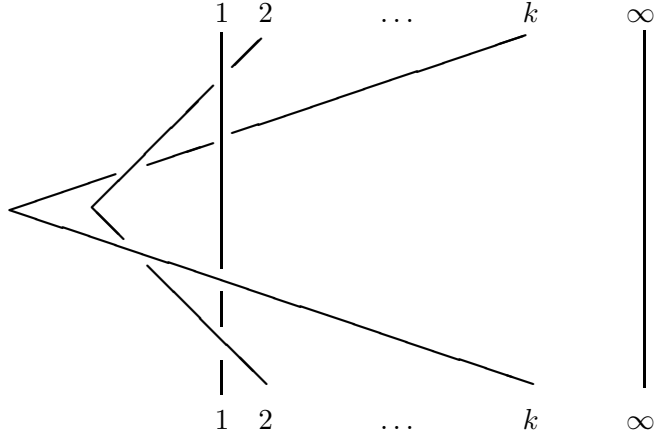
\begin{center}
\begin{picture}(250,180)

\thicklines

\put(250,13){\line(0,1){135}}
\put(90,58){\line(0,1){89}}
\put(90,36){\line(0,1){13}}
\put(90,10){\line(0,1){13}}
\put(10,80){\line(3,-1){198}}
\put(10,80){\line(3,1){40}}
\put(62,97){\line(3,1){25}}
\put(94,109){\line(3,1){111}}
\put(41,81){\line(1,1){46}}
\put(94,134){\line(1,1){11}}
\put(41,81){\line(1,-1){12}}
\put(62,59){\line(1,-1){45}}
\put(87.5,151){1}\put(104,151){2}\put(204,151){$k$}\put(243,151){$\infty$}\put(150,151){$\dots$}

\put(87.5,-3){1}\put(104,-3){2}\put(204,-3){$k$}\put(243,-3){$\infty$}\put(150,-3){$\dots$}
\end{picture}
\end{center}
In the diagram of the pure braid $[\g_{k+1}]$ we can start with $k-1$ intersections of first strand with the second,$
\ldots$, with the $k-$th strand, next add the diagram of $[\g_k]$ (the first strand corresponding to $1$ is in the $k-$th
position, so the word $[\g_k]$ is unchanged) and end with the second intersections of the first strand with the other
 $k-1$ strands. This gives the representation of $[\g_{k+1}]=(\s_1\ldots\s_{k-1})[\g_k](\s_{k-1}\ldots\s_1)$. Hence
$$
    \begin{array}{ll}
    [\g_{k+1}]&= (\s_1\s_2..\s_{k-1})\ldots(\s_1\s_2\s_3)(\s_1\s_2)\s_{1}^{2}(\s_2\s_1)
(\s_3\s_2\s_1)\ldots(\s_{k-1}..\s_2\s_1)\\
       & =\a_{12}(\a_{13}\a_{23})\ldots(\a_{1k}\a_{2k}..\a_{k-1,k})=D_k=\Delta_{k}^2
    \end{array}
$$
(see
Appendix for the last equalities).
\end{proof}

\section{Proofs of Theorem \ref{th:1} and Theorem \ref{th:2}}
The complex Grassmannian  manifolds $Gr^k(\mc P^n)$ are simply connected and the second homotopy group is stable ($n\geq k+1$)
$$\mathbb{Z}\cong\pi_2(\mc P^1)\mathop{\rightarrow}\limits^{\cong}\pi_2(\mc P^{k+1})\mathop{\rightarrow}\limits^{\cong}\pi_2(Gr^{k}(\mc P^{k+1}))\mathop{\rightarrow}\limits^{\cong}\pi_2(Gr^k(\mc P^n)),$$
so we can choose the map $$g:(D^2,S^1)\rightarrow(Gr^k(\mc P^n),L_1)$$
$$g(z)=L_z,\,\,\,\,L_z:(1-|z|)X_0-zX_1=0,X_{k+2}=\ldots=X_n=0$$
as a generator of $\pi_2(Gr^k(\mc P^n))$.

\begin{lemma}\label{l.2}
For $i\leq \min\{k,n-1\}$ we have $$\pi_{1}(\mathcal{F}_{k+1}^{i,i+1})\cong\pi_{1}(\mathcal{F}_{k+1}^{i,n})\,\,\,.$$
\end{lemma}
\begin{proof}
The canonical embedding
$$\mc P^{i+1}\longrightarrow\mc P^n,\,\,\,\,\,[z_0:\ldots:z_i]\mapsto[z_0:\ldots:z_i:0:\ldots:0]$$
induces the following commutative diagram of fibrations
\begin{center}
\begin{picture}(300,100)
\thicklines
\put(215,75){\vector(0,-1){30}}
\put(147,80){\vector(1,0){40}}
\put(127,75){\vector(0,-1){30}}
\put(56,75){\vector(0,-1){30}}
\put(140,35){\vector(1,0){40}}
\put(75,80){\vector(1,0){40}}
\put(75,35){\vector(1,0){40}}
\put(120,80){$\mathcal{F}_{k+1}^{i,i+1}$}
\put(120,30){$\mathcal{F}_{k}^{i,n}$}
\put(190,80){$ Gr^i(\mc P^{i+1})$}
\put(190,30){$Gr^i(\mc P^n)$}
\put(50,80){$\mathcal{F}_{k+1}^{i,i}$}
\put(50,30){$\mathcal{F}_{k+1}^{i,i}$}
\end{picture}
\end{center}
and the result is obtained from the commutative diagram of homotopy groups
\vspace{-0.7cm}
\begin{center}
\begin{picture}(300,120)
\thicklines
\put(230,75){\vector(0,-1){30}}
\put(167,80){\vector(1,0){30}}
\put(137,75){\vector(0,-1){30}}
\put(80,75){\vector(0,-1){30}}
\put(167,35){\vector(1,0){30}}
\put(85,80){\vector(1,0){30}}
\put(85,35){\vector(1,0){30}}
\put(250,80){\vector(1,0){30}}
\put(250,35){\vector(1,0){30}}
\put(120,80){$\pi_1(\mathcal{F}_{k+1}^{i,i}$)}
\put(120,30){$\pi_1(\mathcal{F}_{k+1}^{i,i}$)}
\put(200,80){$\pi_1(\mathcal{F}_{k+1}^{i,i+1}$)}
\put(200,30){$\pi_1(\mathcal{F}_{k+1}^{1,n}$)}
\put(285,76){1}
\put(285,32){1}
 \put(35,80){\vector(1,0){30}}
\put(35,35){\vector(1,0){30}} \put(15,78){$\ldots$}
\put(15,33){$\ldots$} \put(75,80){$\mathbb{Z}$}
\put(73,30){$\mathbb{Z}$} \put(85,55){$\cong$} \put(140,55){$\cong$}
\end{picture}
\end{center}
\vspace{-1.0cm}
\end{proof}
The top dimensional strata are simply connected.
\begin{lemma}\label{l.4}
For any $k$ we have $\pi_{1}(\mathcal{F}_{k}^{k-1,k})\cong\pi_{1}(\mathcal{F}_{k+1}^{k,k})=1$.
\end{lemma}
\begin{proof}
The proof is by induction on $k$ and uses the  fibration of Lemma \ref{l.7}
with contractible fiber $\mc P^k\setminus \mc P^{k-1}$ and the fibration
$$\mathcal{F}_{k}^{k-1,k-1}\hookrightarrow\mathcal{F}_{k}^{k-1,k}\longrightarrow
Gr^{k-1}(\mc P^k)\,\,.$$ For
$k=1$, $\pi_1(\mathcal{F}_{1}^{0,1})=\pi_1(\mc P^1)=1$
 and $\pi_1(\mathcal{F}_{2}^{1,1})=\pi_1(\mathcal{F}_2(\mc P^1))=1$. The later fibration
 gives $\pi_{1}(\mathcal{F}_{k}^{k-1,k})=1$ and the former gives $\pi_{1}(\mathcal{F}_{k+1}^{k,k})=1$.
\end{proof}
The next result covers the simply connectedness cases of Theorem \ref{th:1}.
\begin{proposition}\label{pr:2}
 $\pi_1(\mathcal{F}_{k}^{i,n})=1$ for $i\geq2$ and $k\geq i+1$.
 \end{proposition}
  \begin{proof}
To prove that $\mathcal{F}_{k}^{i,i}$ is simply connected we use
Seifert-Van Kampen theorem \cite{H} for the finite open covering $\mathcal{F}_{k}^{i,i}=\bigcup\big\{\mathcal{F}_{k}^{\mathcal{A},i}\big|\mathcal{A}=\{A\},A\in
{\{1,\ldots,k\} \choose{i+1}}\big\}$. By Lemma \ref{l.4} we have $\pi_{1}(\mathcal{F}_{i+1}^{i,i})=1$ and
using  fibrations as in Lemma \ref{l.6.5} we obtain
$$\pi_1(\mathcal{F}_{k-i-1}(\mc P^i\setminus\{i+1\,\hbox{points}\}))\rightarrow\pi_{1}(\mathcal{F}_{k}^{\{A\},i})\rightarrow \pi_{1}(\mathcal{F}_{i+1}^{i,i})\, ,$$
all the pieces of the covering are simply connected: the fiber $ \mathcal{F}_{k-i-1}(\mc P^i\setminus\{i+1\,\hbox{points}\})$
is simply connected for $i\geq2$ because $\mc P^i\setminus\{i+1\,\hbox{points}\}$ is simply connected (for the special case $k=i+1,\,\mathcal{F}_{k}^{\{A\},i}=\mathcal{F}_{i+1}^{i,i}$). Also, every intersection
$\mathcal{F}_{k}^{\{A_1\},i}\cap\ldots\cap\mathcal{F}_{k}^{\{A_r\},i}=\mathcal{F}_{k}^{\mathcal{A},i}$
(where $\mathcal{A}=\{A_1,\ldots,A_r\}$) is path connected. 
Using the fibration in the proof of Lemma \ref{l.2}, $\pi_1(\mathcal{F}_{k}^{i,i})=1$ implies $\pi_1(\mathcal{F}_{k}^{i,n})=1$ for $n\geq i\geq 2$.
 \end{proof}
Using the geometrical fibrations introduce in previous section we start the inductive proof of Theorem \ref{th:1}.

\subsection{Proof of Theorem \ref{th:1}}
From Lemma \ref{l.4} $\mathcal{F}_{k+1}^{k,k},\mathcal{F}_{k+1}^{k,k+1}$ are simply connected spaces, and for $k\geq1$ $\mathcal{F}_{k+1}^{1,1}=\mathcal{F}_{k+1}(\mc P^1)$, so we have the classical result (see \cite{F})
$$\pi_1(\mathcal{F}_{k+1}^{1,1})=<\a_{ij}\,,\,\,1\leq i<j\leq k\big|(YB3),(YB4),D_{k}^2=1>.$$ Using Proposition \ref{pr:2} we have $\pi_1(\mathcal{F}_{k}^{i,n})=1$ for $n\geq i\geq2, k\geq i+1$ and Lemma \ref{l.2} implies  $\pi_{1}(\mathcal{F}_{k+1}^{1,2})\cong\pi_{1}(\mathcal{F}_{k+1}^{1,n}),\,\,\,(k\geq1$), so we have to compute $\mathcal{F}_{k+1}^{1,2}$. Consider the following  fibration (see Proposition \ref{pr:1})
$$ p:\mathcal{F}_{k+1}^{1,2}\rightarrow Gr^1(\mc P^2)$$
with fiber $\mathcal{F}_{k+1}(\mc P^1)$.
Since $\pi_1(Gr^1(\mc P^2))=1$ and $\pi_2(Gr^1(\mc P^2))=\pi_2(\mc P^2)\cong\mathbb{Z}$, the homotopy exact sequence
$$\ldots\longrightarrow\pi_2(\mathcal{F}_{k+1}^{1,2})\longrightarrow \mathbb{Z}\mathop{\longrightarrow}\limits^{\delta_*}\pi_1(\mathcal{F}_{k+1}(\mc P^1))\longrightarrow\pi_1(\mathcal{F}_{k+1}^{1,2})\longrightarrow 1$$
gives $\pi_1(\mathcal{F}_{k+1}^{1,2})=\pi_1\big(\mathcal{F}_{k+1}(\mc P^1)\big)\big/\big<Im \delta_{*}\big>.$

Let $[g]$ be the generator of $\pi_2(Gr^1(\mc P^2),L_1),\,\,$  $L_1$ is the line $X_1=0$ ,
$$g:(D^2,S^1)\longrightarrow (Gr^1(\mc P^2),L_1),\,\,\,\,z\longmapsto L_z:(1-|z|)X_0-zX_1=0\,.$$

We choose the lift $\widetilde{g}$
\begin{center}
\begin{picture}(360,100)
\thicklines
\put(180,65){\vector(0,-1){30}}
\put(100,25){\vector(1,0){50}}
\put(103,40){\vector(2,1){50}}
\put(55,20){$(D^2,S^1)$}
\put(160,20){$(Gr^i(\mc P^2),L_1)$}
\put(160,75){$\big(\mathcal{F}_{k+1}^{1,2},\mathcal{F}_{k+1}(L_1)\big)$}
\put(184,47){$p$}
\put(120,55){$\widetilde{g}$}
\put(114,15){$g$}
\end{picture}
\end{center}
$\widetilde{g}(z)=([z:1-|z|:1],[z:1-|z|:\frac{1}{2}],\ldots,[z:1-|z|:\frac{1}{k}],[z,1-|z|,0])$ and the image $\delta_{*}[g]$ is given by the restriction $\g=\widetilde{g}\big|_{S^1}:\,\,S^1\longrightarrow \mathcal{F}_{k+1}(\mc P^1)$
$$\g(z)=([z:1],[z:\frac{1}{2}],\ldots,[z:\frac{1}{k}],[1:0])=
([z:1],[2z:1],\ldots,[kz:1],[1:0]).$$
By the Lemma \ref{l.6}, $[\g]=\a_{12}\a_{13}\a_{23}\ldots\a_{1k}\ldots\a_{k-1,k}=D_{k}$, hence we obtain
$$
  \begin{array}{ll}
    \pi_1(\mathcal{F}_{k+1}^{1,2})&=\pi_1\big(\mathcal{F}_{k+1}(\mc P^1)\big)\big/\big<Im \delta_{*}\big> \\
    &=\big<\a_{ij}\,,\,i\leq i<j\leq k\big|(YB3)_k,(YB4)_k,D_{k}^2=1,D_{k}=1\big> \\
    &=\big<\a_{ij}\,,\,i\leq i<j\leq k\big|(YB3)_k,(YB4)_k,\,D_{k}=1\big>.\,\,\,\,\,\,\,\,\,\,\,\,\,\,\,\, \square
    \end{array}
$$

\noindent\textbf{Proof of Corollary\ref{l.1.2}}.
In $\pi_1(\mathcal{F}_{k+1}^{1,1})_{ab}$ we have ${k\choose{2}}$ generators $\big(a_{ij}\big)_{1\leq i<j\leq k}$ and one relation: $2(a_{12}+a_{13}+a_{23}+\ldots+a_{k-1,k})=0$. Changing the last generator $a_{k-1,k}$ with $a=a_{12}+a_{13}+a_{23}+\ldots+a_{k-1,k}$ we obtain a presentation of $\mathbb{Z}^{{k\choose 2}-1}\oplus \mathbb{Z}_2$. In the second case $a_{k-1,k}$ can be eliminated.\hspace{2.4cm}$\square$
\begin{example}\label{ex:1}
For $k=2$ we have $\pi_1(\mathcal{F}_{3}^{1,2})=1,\,\pi_2(\mathcal{F}_{3}^{1,2})\cong\mathbb{Z}$; for $k=3$, $\pi_1(\mathcal{F}_{4}^{1,2})\cong\mathbb{F}_2.$
\end{example}
For $k=2$, $\pi_2(\mathcal{F}_3(\mc P^1))=1$ (see \cite{B2}). Using $\pi_2(Gr^1(\mc P^2))\cong\mathbb{Z}$ and the fibration
$\mathcal{F}_3(\mc P^1)\rightarrow \mathcal{F}_{3}^{1,2}\rightarrow Gr^1(\mc P^2)$, we have
$$0\longrightarrow\pi_2(\mathcal{F}_{3}^{1,2}) \longrightarrow \mathbb{Z}\mathop{\longrightarrow}\limits^{\delta_*}\mathbb{Z}_2\longrightarrow\pi_1(\mathcal{F}_{3}^{1,2})\longrightarrow1\,\,;$$
since the image of $\delta_*$ contains $D_2=\a_{12}$ (see Lemma \ref{l.6}), $\delta_{*}$ is surjective and  $\pi_1(\mathcal{F}_{3}^{1,2})=1$ and also $\pi_2(\mathcal{F}_{3}^{1,2})=\mathbb{Z}$.

 For $k=3,\,\,\pi_1(\mathcal{F}_{4}^{1,2})$ is generated by $\a_{12},\a_{13},\a_{23}$ with defining relations  $$\a_{12}\a_{13}\a_{23}=\a_{13}\a_{23}\a_{12}=\a_{23}\a_{12}\a_{13}=1,$$
therefore any of $\a_{ij} $ can be discarded: $\pi_1(\mathcal{F}_{4}^{1,2})\cong\mathbb{F}(\a_{12},\a_{23})\cong\mathbb{F}(\a_{13},\a_{23})\cong\mathbb{F}(\a_{12},\a_{13})$.\hspace{10.1cm} $\square$\\
\begin{remark}
For any continuous maps $f,g:S^2\rightarrow S^2$ without fixed points, there is an $x\in S^2$ such that $f(x)=g(x)$. More generally, for any map $F: S^2\rightarrow\mathcal{F}_3(S^2)$, the three projections $pr_j\circ F$ are homotopically trivial.
\end{remark}
\begin{proof}
The composition $\pi_2(S^2)\mathop{\longrightarrow}\limits_{F_*}\pi_2(\mathcal{F}_3(S^2))\mathop{\longrightarrow}\limits_{pr_*}\pi_2(S^2)$ is trivial.
\end{proof}

\subsection{Proof of Theorem \ref{th:2}} First we have
$\pi_{1}(\mathcal{C}_{k}^{i,n})\cong\Sigma_{k}$ for $2\leq i\leq \min(k+1,n)$ as
a consequence of Theorem \ref{th:1} and the regular covering
$p:\mathcal{F}_{k}^{i,n}\longrightarrow\mathcal{C}_{k}^{i,n}$. Next we have
$$\pi_1(\mathcal{C}_{k+1}^{1,1})=\pi_1(\mathcal{C}_{k+1}(\mc P^1))\cong\mathcal{B}_{k+1}/\big<\s_1\ldots\s_{k-1}\s_{k}^{2}\s_{k-1}\ldots\s_1=1\big>.$$
Finally we have to compute  $\pi_1(\mathcal{C}_{k+1}^{1,n})$ for $n,k\geq2$. The following commutative diagram of the fibrations and coverings
\begin{center}
\begin{picture}(280,150)
\thicklines
\put(215,65){\vector(0,-1){30}}
\put(215,115){\vector(0,-1){30}}
\put(207,120){$\{*\}$}
\put(147,70){\vector(1,0){40}}
\put(127,115){\vector(0,-1){30}}
\put(124,120){$\Sigma_{k+1}$}
\put(127,65){\vector(0,-1){30}}
\put(56,65){\vector(0,-1){30}}
\put(56,115){\vector(0,-1){30}}
\put(51,120){$\Sigma_{k+1}$}
\put(140,25){\vector(1,0){40}}

\put(154,120){\vector(1,0){40}}
\put(78,120){\vector(1,0){40}}

\put(75,70){\vector(1,0){40}}
\put(75,25){\vector(1,0){40}}
\put(120,70){$\mathcal{F}_{k+1}^{1,n}$}
\put(120,20){$\mathcal{C}_{k+1}^{1,n}$}
\put(190,70){$ Gr^1(\mc P^{n})$}
\put(190,20){$Gr^1(\mc P^n)$}
\put(50,70){$\mathcal{F}_{k+1}^{1,1}$}
\put(50,20){$\mathcal{C}_{k+1}^{1,1}$}

\end{picture}
\end{center}
induces the commutative diagram of homotopy groups
\begin{center}
\begin{picture}(330,180)
\thicklines

 \put(215,170){$1$}
\put(124,170){$1$}
\put(127,165){\vector(0,-1){20}}
\put(215,165){\vector(0,-1){20}}
\put(127,75){\vector(0,-1){30}}
\put(120,35){$\Sigma_{k+1}$}
\put(210,35){$\Sigma_{k+1}$}
  \put(127,28){\vector(0,-1){20}}
  \put(217,28){\vector(0,-1){20}}
   \put(125,-3){$1$}

   \put(215,-3){$1$}
\put(127,125){\vector(0,-1){30}}
\put(65,125){\vector(0,-1){30}}
\put(219,75){\vector(0,-1){30}}
\put(220,125){\vector(0,-1){30}}
\put(160,130){\vector(1,0){30}}
\put(160,85){\vector(1,0){30}}
\put(250,85){\vector(1,0){30}}
\put(250,130){\vector(1,0){30}}
\put(77,130){\vector(1,0){30}}
\put(85,135){$\delta_*$}
\put(85,90){$\delta_{*}'$}
\put(77,85){\vector(1,0){30}}
 \put(25,130){\vector(1,0){30}}
 \put(25,85){\vector(1,0){30}}
  \put(8,130){$\ldots$}
  \put(8,85){$\ldots$}
\put(200,130){$\pi_1(\mathcal{F}_{k+1}^{1,n})$}
\put(200,80){$\pi_1(\mathcal{C}_{k+1}^{1,n}$)}
\put(286,125){$ 1$}
\put(286,81){$1$}
\put(115,130){$\pi_1(\mathcal{F}_{k+1}^{1,1}$)}
\put(115,80){$\pi_1(\mathcal{C}_{k+1}^{1,1}$)}
\put(63,127){$\mathbb{Z}$}
\put(63,82){$\mathbb{Z}$}
\put(70,110){$\cong$}
\put(155,37){\vector(1,0){45}}
\end{picture}
\end{center}

Since $Im\delta_*=<D_k>$, the left square  gives $Im\delta_{*}'=<\Delta_{k}^2>$, therefore
$$\pi_1(\mathcal{C}_{k+1}^{1,n})\cong\mathcal{B}_{k+1}\big/\big<\s_1\ldots\s_{k-1}\s_{k}^{2}\s_{k-1}\ldots\s_1,(\s_1\s_2\s_1\ldots\s_{k-1}\ldots\s_1)^2\big>.\,\,\,\,\,\,\,\,\,\square$$

\textbf{Proof of Corollary \ref{l.1.4}}. The abelianized group $\pi_1(\mathcal{C}_{k+1}^{1,1})_{ab}$ is generated by $\big(s_{i}\big)_{i=1,\ldots,k}$ and the relations $s_{i+1}=s_i$ and $2(s_{1}+\ldots+s_k)=0$ and this gives a presentation of $\mathbb{Z}_{2k}$. The group $\pi_1(\mathcal{C}_{k+1}^{1,n})_{ab}$ has again one generator $s$ and the relations $2ks=0$ and $k(k-1)s=0$. The greatest common divisor of $2k$ and $k(k-1)$ is $k$ for $k$ even and $2k$ for $k$ odd, hence the result.\hspace{0.7cm}$\square$
\section{Pappus' Configurations}
Let us define the {\em{space of Pappus' configurations}} (in $\mc P^2$) by\\
$\mathcal{P}=\big\{(A_1,B_1,C_1,A_2,B_2,C_2)\in\mathcal{F}_{6}^{2,2}\big|(A_i,B_i,C_i)\in\mathcal{F}_{3}^{1,2},A_i,B_i,C_i\neq I\big\}$\\
(here $I=d_1\cap d_2$, where $d_i$ is the line containing $A_i,B_i,C_i,\,\,i=1,2$)
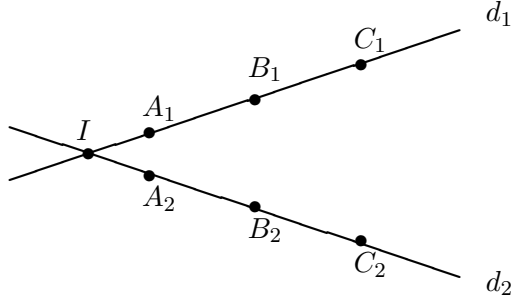
\begin{center}
\begin{picture}(250,180)
\thicklines
\put(10,80){\line(3,-1){170}}
\put(10,60){\line(3,1){170}}
\put(60,84){$A_1$} \put(100,99){$B_1$} \put(140,110){$C_1$} \put(190,120){$d_1$}
\put(60,75){$\bullet$} \put(100,87.5){$\bullet$} \put(140,101){$\bullet$}
\put(60,50){$A_2$} \put(100,38){$B_2$} \put(140,25){$C_2$} \put(190,18){$d_2$}
\put(60,59){$\bullet$} \put(100,47){$\bullet$} \put(140,34){$\bullet$}
\put(37,67){$\bullet$}
\put(35,75){$I$}
\end{picture}
\end{center}
 and also the{ \em{space of Pappus' configurations with a fixed intersection point}} $I\in\mc P^2$ as
 $$\mathcal{P}_{I}=\big\{(A_1,B_1,C_1,A_2,B_2,C_2)\in\mathcal{P}\big|d_1\cap d_2=I\big\}.$$

 In order to find the fundamental group of $\mathcal{P}$ and $\mathcal{P}_I$ we apply the same method: define two fibrations and use their homotopy exact sequences.
\begin{lemma}\label{p1}
The following projection
$$\Gamma:\mathcal{P}\rightarrow \mc P^2,\,(A_1,B_1,C_1,A_2,B_2,C_2)\mapsto I=d_1\cap d_2$$
is a locally trivial fibration with fiber $\mathcal{P}_I.$
\end{lemma}
\begin{proof}
Choose a line $l\subset \mc P^2\setminus\{I^0\}$ and the neighborhood $\mathcal{U}_l=\mc P^2\setminus l$ of $I^0$.
For a point $I$ in this neighborhood and a Pappus' configuration $(A_{1}^0,B_{1}^0,C_{1}^0,A_{2}^0,B_{2}^0,C_{2}^0)$ on two lines $d_{1}^0,d_{2}^0$ containing $I^0$, construct lines $d_1,d_2$ and the configuration $(A_1,B_1,C_1,A_2,B_2,C_2)$ as follows: consider the points $D_i=l\cap d_{i}^0$ and $Q=l\cap II^0$ and define $d_i=ID_i,\,A_i=d_i\cap QA_{i}^0$ and in the same way $B_i,C_i\,(i=1,2)$. We describe this construction using coordinates to show that the map
$$(A_{1}^0,B_{1}^0,C_{1}^0,A_{2}^0,B_{2}^0,C_{2}^0)\mapsto(A_1,B_1,C_1,A_2,B_2,C_2)$$ has a continuous extension on the singular locus $(d_{1}^0\cup d_{2}^0)\setminus l$. Choose a projective frame such that $I^0=[0:0:1],\, l:X_2=0$. If $I=[s:t:1]$ and $A_{i}^0=[n_i:-m_i:a_i],\,B_{i}^0=[n_i:-m_i:b_i],\,C_{i}^0=[n_i:-m_i:c_i]$ ($a_i,b_i,c_i$ and $0$ are distinct and also $m_1n_2\neq m_2n_1$), then we define $A_{i}=[n_i+sa_i:-m_i+ta_i:a_i],\,B_{i}=[n_i+sb_i:-m_i+tb_i:b_i],\,C_{i}=[n_i+sc_i:-m_i+tc_i:c_i]\,(i=1,2)$, and these formulae agree with the geometrical construction given for nondegenerate positions of $I\in \mc P^2\setminus (d_1\cup d_2\cup l)$. The trivialization over $\mathcal{U}_l$ is given by
$$\phi:\mathcal{U}_l\times \mathcal{P}_{I^0}\rightarrow\Gamma^{-1}(\mathcal{P}_{I^0})$$
$$\phi(I,(A_{1}^0,B_{1}^0,C_{1}^0,A_{2}^0,B_{2}^0,C_{2}^0))=(A_1,B_1,C_1,A_2,B_2,C_2).$$
\end{proof}
\begin{lemma}\label{p2}
The projection $$\Lambda:\mathcal{P}_{I}\rightarrow\mathcal{F}_{2}(\mc P^1),\,(A_1,B_1,C_1,A_2,B_2,C_2)\mapsto(d_1,d_2)$$ is a locally trivial fibration with fiber $\mathcal{F}_3(\mc )\times\mathcal{F}_{3}(\mc )$.
\end{lemma}
\begin{proof}
Choose a point $Q$ in $\mc P^2\setminus(d_{1}^0\cup d_{2}^0)$ and the neighborhood $\mathcal{U}_Q=\{(d_1,d_2)\in\mathcal{F}_{2}(\mc P^1)\big|Q\notin d_1\cup d_2\}$. The trivialization over $\mathcal{U}_Q$ is given by
$$\psi:\mathcal{U}_Q\times \mathcal{F}_{3}(d_{1}^0)\times \mathcal{F}_{3}(d_{2}^0)\rightarrow \Lambda^{-1}(\mathcal{U}_Q)$$
$$\psi\big((d_1,d_2),(A_{1}^0,B_{1}^0,C_{1}^0),(A_{2}^0,B_{2}^0,C_{2}^0)\big)=(A_1,B_1,C_1,A_2,B_2,C_2),$$
where $A_i=d_i\cap QA_{i}^0$ and similarly for $B_i,C_i\,(i=1,2)$. Obviously, $A_i,B_i,C_i$ and $I$ are four distinct points on $d_i$.
\end{proof}
Choose in $\mathcal{P}_{I=[0:0:1]}$ the base point $p_0=(A_{1}^0,B_{1}^0,C_{1}^0,A_{2}^0,B_{2}^0,C_{2}^0)$ where
$A_{1}^0=[1:0:1],B_{1}^0=[2:0:1],C_{1}^0=[3:0:1],A_{2}^0=[0:1:1],B_{2}^0=[0:2:1],C_{2}^0=[0:3:1]$.
\begin{proposition}
The fundamental group of $\mathcal{P}_I$ is given by
$$\pi_1(\mathcal{P}_I,p_0)\cong\mathcal{PB}_3\times\mathbb{F}_2.$$
\end{proposition}
\begin{proof}
Lemma \ref{p2} gives the exact sequence
$$\ldots\rightarrow\pi_{2}(\mathcal{F}_{2}(\mc P^1))\mathop{\longrightarrow}\limits^{\delta_*}\pi_1(\mathcal{F}_3(\mc )\times\mathcal{F}_{3}(\mc ))\longrightarrow\pi_1(\mathcal{P_I})\longrightarrow1,$$
where the first group is cyclic generated by the homotopy class of the map
$$\l:(D^2,S^1)\rightarrow\mathcal{F}_{2}(\mc P^1),\,z\mapsto\big([1-|z|:z],[\overline{z}:|z|-1]\big)$$
(that means the two lines through $I=[0:0:1]$ are $d_1(z):(1-|z|)X_0+zX_1=0$ and $d_2(z):\overline{z}X_0+(|z|-1)X_1=0$). The second group is generated by $\{\a_{ij},\a'_{ij}\}_{1\leq i<j\leq3}$ where
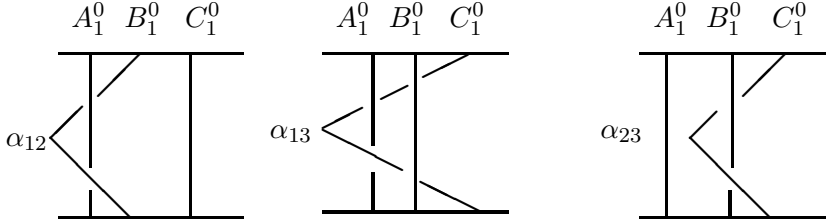
\begin{center}
\begin{picture}(360,100)
\thicklines
\put(30,65){\line(1,0){70}}
\put(30,3){\line(1,0){70}}
\put(42,22){\line(0,1){43}}
\put(45,49){\line(1,1){16}}
\put(27,33){\line(1,1){12}}
\put(57,3){\line(-1,1){30}}
\put(80,3){\line(0,1){62}}
\put(42,3){\line(0,1){10}}

\put(10,30){$\a_{12}$}

\put(35,75){$A_1^0$}
\put(55,75){$B_1^0$}
\put(78,75){$C_1^0$}

\put(130,65){\line(1,0){70}}
\put(130,5){\line(1,0){70}}
\put(149,30){\line(0,1){35}}
\put(168,56){\line(2,1){17}}
\put(152,48){\line(2,1){10}}
\put(130,37){\line(2,1){15}}

\put(149,5){\line(0,1){15}}
\put(129,22){\line(1,1){4}}
\put(160,21){\line(-2,1){30}}
\put(190,5){\line(-2,1){23}}
\put(165,5){\line(0,1){60}}
\put(110,33){$\a_{13}$}

\put(135,75){$A_1^0$}
\put(155,75){$B_1^0$}
\put(178,75){$C_1^0$}

\put(250,65){\line(1,0){70}}
\put(250,3){\line(1,0){70}}
\put(285,22){\line(0,1){43}}
\put(289,49){\line(1,1){16}}
\put(269,33){\line(1,1){10}}
\put(299,3){\line(-1,1){30}}
\put(260,3){\line(0,1){62}}
\put(284,3){\line(0,1){10}}
\put(235,33){$\a_{23}$}

\put(255,75){$A_1^0$}
\put(275,75){$B_1^0$}
\put(300,75){$C_1^0$}

\end{picture}
\end{center}

\noindent the points moving on the line $d_{1}^{0}:X_1=0$, respectively
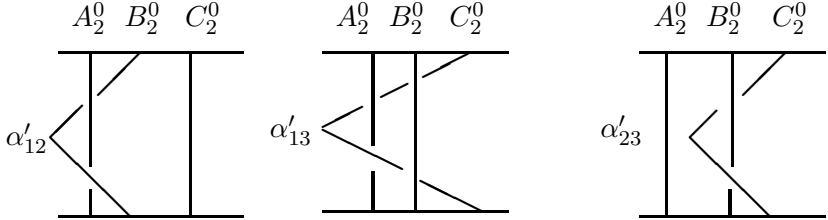
\begin{center}
\begin{picture}(360,100)
\thicklines
\put(30,65){\line(1,0){70}}
\put(30,3){\line(1,0){70}}
\put(42,22){\line(0,1){43}}
\put(45,49){\line(1,1){16}}
\put(27,33){\line(1,1){12}}
\put(57,3){\line(-1,1){30}}
\put(80,3){\line(0,1){62}}
\put(42,3){\line(0,1){10}}

\put(10,30){$\a'_{12}$}

\put(35,75){$A_2^0$}
\put(55,75){$B_2^0$}
\put(78,75){$C_2^0$}

\put(130,65){\line(1,0){70}}
\put(130,5){\line(1,0){70}}
\put(149,30){\line(0,1){35}}
\put(168,56){\line(2,1){17}}
\put(152,48){\line(2,1){10}}
\put(130,37){\line(2,1){15}}

\put(149,5){\line(0,1){15}}
\put(129,22){\line(1,1){4}}
\put(160,21){\line(-2,1){30}}
\put(190,5){\line(-2,1){23}}
\put(165,5){\line(0,1){60}}
\put(110,33){$\a'_{13}$}

\put(135,75){$A_2^0$}
\put(155,75){$B_2^0$}
\put(178,75){$C_2^0$}

\put(250,65){\line(1,0){70}}
\put(250,3){\line(1,0){70}}
\put(285,22){\line(0,1){43}}
\put(289,49){\line(1,1){16}}
\put(269,33){\line(1,1){10}}
\put(299,3){\line(-1,1){30}}
\put(260,3){\line(0,1){62}}
\put(284,3){\line(0,1){10}}
\put(235,33){$\a'_{23}$}

\put(255,75){$A_2^0$}
\put(275,75){$B_2^0$}
\put(300,75){$C_2^0$}

\end{picture}
\end{center}

\noindent the points moving on the line $d_{2}^{0}:X_0=0$ and it has the presentation
$$\pi_1(\mathcal{F}_3(\mc )\times\mathcal{F}_{3}(\mc ))\cong\big<\a_{ij},\a'_{ij}, 1\leq i<j\leq3\big|(YB3)_{\a_{ij}},(YB3)_{\a'_{ij}},[\a_{ij},\a'_{pq}]=1\big>.$$
We choose the lift $\widetilde{\lambda}:(D^2,S^1)\rightarrow\big(\mathcal{P}_I,\mathcal{F}_3(d_1^0\setminus I^0)\times\mathcal{F}_3(d_2^0\setminus I^0)\big),$

$z\mapsto\big([z:|z|-1:1],[z:|z|-1:\frac{1}{2}],[z:|z|-1:\frac{1}{3}],[1-|z|:\overline{z}:1],[1-|z|:
\overline{z}:\frac{1}{2}],[1-|z|:\overline{z}:\frac{1}{3}]\big)$, hence $Im \delta_*$ is generated by the class of the map
$$\widetilde{\l}\big|_{S^1}:S^1\rightarrow\mathcal{F}_3(d_1^0\setminus I^0)\times\mathcal{F}_3(d_2^0\setminus I^0),\,z\mapsto\big((z,2z,3z),(\overline{z},2\overline{z},3\overline{z})\big),$$
and this is $(D_3,(D'_3)^{-1})$ (see lemma \ref{l.6}). The group $\pi_{1}(\mathcal{P}_I,p_0)$ is generated by $\b_{ij},\b'_{pq}$, the images of $\a_{ij},\a'_{pq}$, with relations $(YB3)_{\b_{ij}},(YB3)_{\b'_{ij}}$, $[\b_{ij},\b'_{pq}]=1$ and $\b'_{12}\b'_{13}\b'_{23}=\b_{12}\b_{13}\b_{23}$. $\b'_{23}$ can be eliminated and we find the relations $(YB3)_{\b_{ij}}, [\b_{ij},\b'_{pq}]$, and this gives the presentation of the direct product $\mathcal{PB}_3\times \mathbb{F}_2$.
\end{proof}
\begin{remark}
The generator $\b_{12}$ of $\pi_{1}(\mathcal{P}_I)$ is given by the loop
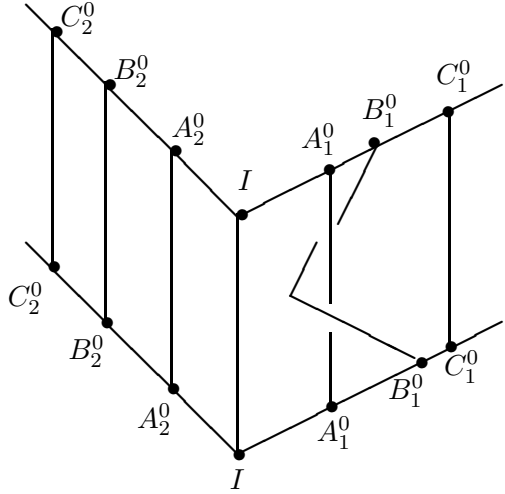
\begin{center}
\begin{picture}(360,210)
\thicklines
\put(150,105){\line(2,1){100}}
\put(150,105){\line(-1,1){80}}

\put(150,15){\line(2,1){100}}
\put(150,15){\line(-1,1){80}}

\put(150,15){\line(0,1){91}}\put(150,115){$I$}\put(149,103){$\bullet$}\put(147,2){$I$}\put(148,12){$\bullet$}
\put(125,40){\line(0,1){90}}\put(125,135){$A_2^0$}\put(124,127){$\bullet$}\put(112,25){$A_2^0$}\put(123,37){$\bullet$}
\put(100,66){\line(0,1){90}}\put(103,157){$B_2^0$}\put(99,152){$\bullet$}\put(86,51){$B_2^0$}\put(98,62){$\bullet$}
\put(80,86){\line(0,1){90}}\put(83,177){$C_2^0$}\put(79,172){$\bullet$}\put(63,71){$C_2^0$}\put(78,83){$\bullet$}

\put(230,56){\line(0,1){90}}
\put(185,72){\line(0,1){52}}
\put(188,102){\line(1,2){15}}
\put(170,75){\line(1,2){10}}
\put(170,75){\line(2,-1){47}}
\put(185,32){\line(0,1){29}}

\put(225,155){$C_1^0$}\put(227,142){$\bullet$}\put(228,45){$C_1^0$}\put(228,53){$\bullet$}
\put(197,142){$B_1^0$}\put(199,130){$\bullet$}\put(207,36){$B_1^0$}\put(217,47){$\bullet$}
\put(174,131){$A_1^0$}\put(182,120){$\bullet$}\put(180,20){$A_1^0$}\put(183,30){$\bullet$}
\end{picture}
\end{center}
and there are similar pictures for $\b_{13},\b_{23},\b'_{12},\b'_{13}$.
\end{remark}
\noindent\textbf{Proof of Theorem\ref{th:3}}. Lemma \ref{p1} gives the exact sequence
$$\ldots\longrightarrow\pi_{2}(\mc P^2)\mathop{\longrightarrow}\limits^{\delta_*}\pi_1(\mathcal{P}_I)\longrightarrow\pi_1(\mathcal{P})\longrightarrow1,$$
where the first group is cyclic generated by the homotopy class of the map
$$\g:(D^2,S^1)\rightarrow(\mc P^2,[0:0:1]),\,z\mapsto[0:1-|z|:z].$$
We choose the lift $\widetilde{\gamma}:(D^2,S^1)\rightarrow(\mathcal{P},\mathcal{P}_I)$,\,\,\,$\widetilde{\gamma}(z)=\big([1:0:1],[2:|z|-1:2-z],[3:2|z|-2:3-2z],[0:1-|z|+\overline{z}:|z|-1+z],[0:1-|z|+2\overline{z}:2|z|-2+z],[0:1-|z|+3\overline{z}:3|z|-3+z]\big)$, hence $Im \delta_*$ is generated by the class of the map $\widetilde{\g}\big|_{S^1}:S^1\rightarrow \mathcal{P}_I$, $z\mapsto\big([1:0:1],[\frac{2}{2-z}:0:1],[\frac{3}{3-2z}:0:1],[0:\overline{z}^2:1],[0:2\overline{z}^2:1],[0:3\overline{z}^2:1]\big)$; the  images of the loops $z\mapsto\frac{2}{2-z},\,z\mapsto\frac{3}{3-z}$ are two cirles around the constant loop $z\mapsto1$ and the first circle is included in the second one. As in Lemma \ref{l.6} we find  $\delta_*[\widetilde{\g}]=(\b_{12}\b_{13}\b_{23})(\b'_{12}\b'_{13}\b'_{23})^{-2}=(\b_{12}\b_{13}\b_{23})^{-1}$, therefore $\b_{23}$ can be eliminated and this gives the presentation of the direct product $\mathbb{F}_2\times \mathbb{F}_2$ (the generators are $\b_{12},\b_{13}$ and $\b'_{12},\b'_{13}$).\hspace{6.2cm}$\square$\\

Define the {\em{space of Pappus' configurations}} (in $\mc P^{n}$) by\\
$\mathcal{P}^{(n)}=\big\{(A_1,B_1,C_1,A_2,B_2,C_2)\in\mathcal{F}_{6}^{2,n}\big|(A_i,B_i,C_i)\in\mathcal{F}_{3}^{1,n},A_i,B_i,C_i\neq I\big\}$\\
(here $I=d_1\cap d_2$, where $d_i$ is the line containing $A_i,B_i,C_i,\,\,i=1,2$).
\begin{theorem}
$\pi_1(\mathcal{P}^{(n)})\cong\pi_1(\mathcal{P})\cong\mathbb{F}_2\times\mathbb{F}_2$.
\end{theorem}
\begin{proof}
Using the fibration (the proof is like in Proposition \ref{pr:1})
$$\mathcal{P}=\mathcal{P}^{(2)}\longrightarrow\mathcal{P}^{(n)}\longrightarrow Gr^2(\mc P^n)$$
first we find that $\pi_1(\mathcal{P}^{(3)})\cong\pi_1(\mathcal{P}^{(n)})$ for $n\geq3$ (like in Lemma \ref{l.2}) and secondly, in the fibration
$\mathcal{P}\longrightarrow\mathcal{P}^3\longrightarrow Gr^2(\mc P^3)$, the boundary morphism $\pi_2(Gr^2(\mc P^3))\rightarrow \pi_1(\mathcal{P})$ is trivial: take as generator of $\pi_{2}(Gr^2(\mc P^2))$ the class of $\mu:(D^2,S^1)\rightarrow \big(Gr^2(\mc P^3),H_1\big),\,z\mapsto H_z:(1-|z|)X_0+zX_3=0$, and choose the lift $\widetilde{\mu}:(D^2,S^1)\rightarrow(\mathcal{P}^{(3)},\mathcal{P}),\,z\mapsto\big([z:0:1:|z|-1],[2z:0:1:2|z|-2],[3z:0:1:3|z|-3],[0:1:1:0],[0:2:1:0],[0:3:1:0]\big)$, and $\delta_{*}[\widetilde{\mu}|_{S^1}]=\b_{12}\b_{13}\b_{23}=1$.

\end{proof}
\section{appendix}
In this Appendix we recollect different formulae in braid groups involving the braids $\s_1\ldots\s_k,\,\,\,\s_k\ldots\s_1$ (consecutive factors) and Garside braid $\Delta_n$ and also direct proofs of these relations scattered in the literature ([B2,G,M1]).
\begin{lemma}\label{rmk.1}
\begin{enumerate}
  \item $(\s_1\ldots\s_k)\s_i=\s_{i+1}(\s_1\ldots\s_k)$ for $i=1,\ldots,k-1$ ;
  \item $(\s_k\ldots\s_{1})\s_i=\s_{i-1}(\s_k\ldots\s_1)$ for $i=2,\ldots,k$ ;
       \item$(\s_k\ldots\s_1\s_1\ldots\s_k)\s_j=\s_j(\s_k\ldots\s_1\s_1\ldots\s_k),\,\,\,\hbox{for}\,\,j=1,\ldots,k-1$ .
\end{enumerate}
\end{lemma}
\begin{proof}
These equalities are direct consequences of braid relations
$$\s_i\s_j=\s_j\s_i,\,\,\,\,\,\,\forall\,\,\,i,j=1,\ldots,n-1\,\,\,\hbox{with}\,\,|i-j|\geq 2$$
$$\s_i\s_{i+1}\s_i=\s_{i+1}\s_i\s_{i+1}\,\,\,\,\hbox{for}\,\,\,i=1,\ldots,n-2.$$
\end{proof}
In the next Lemma there are given some symmetric words in the diagram of Garside braid $\Delta_n$; the first one is the smallest in the length-lexicographic order, the last one is the greatest.
\begin{lemma}\label{l.1}
We have the following representations of the fundamental braid $\Delta_n$
$$
  \begin{array}{ll}
    \Delta_{n}&=\s_1(\s_2\s_1)\ldots(\s_{n-1}\ldots\s_1) \\
   &=(\s_1\ldots\s_{n-1})\ldots(\s_1\s_2)\s_1\\
  &=(\s_{n-1}\s_{n-2}\ldots\s_1)(\s_{n-1}\s_{n-2}\ldots\s_2)\ldots(\s_{n-1}\s_{n-2})\s_{n-1}\\
  &=\s_{n-1}(\s_{n-2}\s_{n-1})(\s_{n-3}\s_{n-2}\s_{n-1})\ldots(\s_1\ldots\s_{n-1}).
  \end{array}
$$
\end{lemma}
\noindent\textit{Proof}.
Let us denote $\Delta_{n}',\Delta_{n}''$ and $\Delta_{n}'''$ the second, the third and the fourth word respectively. We will prove by induction
$\,\Delta_{n}'=\Delta_{n}''',\,\,\,\Delta_{n}''=\Delta_{n},\,\,\,\Delta'''_{n}=\Delta_{n}$. We will denote the {\em shift} of a word $\omega$ by $\Sigma\omega$ (for example, if $\omega=x_1x_2x_4x_3,\,\Sigma\omega=x_2x_3x_5x_4)$.

For $n=3,\,\,\Delta_3=\Delta'_3=\s_1\s_2\s_1=\s_2\s_1\s_2=\Delta_{3}'''=\Delta_{3}''$. Suppose that for $k<n$, we have $\Delta_{k}=\Delta_{k}'=\Delta_{k}''=\Delta_{k}'''$. Applying Lemma (\ref{rmk.1},(a)) on $(\s_1\ldots\s_{n-n})\ldots(\s_1\s_2)\s_1$
we get

$$
    \begin{array}{ll}
    \Delta_{n}'&= (\s_1\ldots\s_{n-1})\ldots(\s_1\s_2)\s_1\\
       &= (\s_2\ldots\s_{n-1})\ldots(\s_2\s_3)\s_2(\s_1\ldots\s_{n-1}) \\
       &=(\Sigma\Delta'_{n-1})(\s_1\ldots\s_{n-1})  \\
      & =(\Sigma\Delta'''_{n-1})(\s_1\ldots\s_{n-1})=\Delta'''_n.
    \end{array}
$$
 Applying Lemma (\ref{rmk.1},(b)) on $(\s_{n-1}\s_{n-2}\ldots\s_2)\ldots(\s_{n-1}\s_{n-2})\s_{n-1}$
we get
$$
\begin{array}{ll}
\Delta_{n}''& =(\s_{n-1}\s_{n-2}\ldots\s_1)(\s_{n-1}\s_{n-2}\ldots\s_2)\ldots(\s_{n-1}\s_{n-2})\s_{n-1} \\
      & = (\s_{n-2}\ldots\s_1)\ldots(\s_{n-2}\s_{n-3})\s_{n-2}(\s_{n-1}\ldots\s_1)\\
       & =\Delta''_{n-1}(\s_{n-1}\ldots\s_1)=\Delta_{n-1}(\s_{n-1}\ldots\s_1)=\Delta_n.
     \end{array}
$$
Finally, use the symmetry given by conjugation with $\Delta_n$ ($\Delta_n\s_i=\s_{n-i}\Delta_n$):
$$\Delta_n=\Delta_n\s_1(\s_1\s_2)\ldots(\s_{n-1}\ldots\s_1)
\Delta_{n}^{-1}=\Delta_{n}'''.\hspace{4cm}\square$$

Now we give some formulae for $\Delta_{n}^2$ using  braid generators $\s_i$ (the first one gives the smallest word ) and also pure braid generators $\a_{ij}$ (the last one is the smallest in the length lexicographic order given by $\a_{12}<\a_{13}<\ldots<\a_{n-1,n}$).
\begin{lemma}\label{l.21}
We have the following representations of $\Delta_{n}^2$
$$
  \begin{array}{ll}
\Delta_{n}^2  &=\s_{1}^2(\s_2\s_{1}^2\s_2)\ldots(\s_{n-1}\s_{n-2}\ldots\s_2\s_{1}^2\s_2\ldots\s_{n-2}\s_{n-1})\\
&=(\s_{n-1}\ldots\s_2\s_{1}^2\s_2\ldots\s_{n-1})\ldots(\s_3\s_2\s_{1}^2\s_2\s_3)(\s_2\s_{1}^2\s_2)\s_{1}^2\\

&=(\s_1\s_2\ldots\s_{n-1})\ldots(\s_1\s_2\s_3)(\s_1\s_2)\s_{1}^{2}(\s_2\s_1)
           (\s_3\s_2\s_1)\ldots(\s_{n-1}\ldots\s_2\s_1)\\
           &=(\s_1\ldots\s_{n-1})^n=(\s_{n-1}\ldots\s_1)^n\\
    &=(\a_{1n}\a_{2n}\ldots\a_{n-1,n})(\a_{1n-1}
\ldots\a_{n-2,n-1})\ldots(\a_{13}\a_{23})\a_{12} \\
    &=\a_{12}(\a_{13}\a_{23})\ldots(\a_{1n}\a_{2n}\ldots\a_{n-1,n}).
  \end{array}
$$
\end{lemma}
\noindent\textit{Proof}. Let $A_n,B_n,C_n,D_n,D'_n,E_n,F_n$ be the above seven words respectively.
\begin{eqnarray*}
A_n&=&\s_{1}^2(\s_2\s_{1}^2\s_2)\ldots(\s_{n-1}\s_{n-2}\ldots\s_2\s_{1}^2\s_2\ldots\s_{n-2}\s_{n-1})\\
&=&\Delta_{n-1}^2(\s_{n-1}\ldots\s_2\s_1)(\s_1\ldots\s_{n-1})\\
&=&\Delta_{n-1}(\Delta_n)(\s_1\ldots\s_{n-1})=\Delta_{n-1}(\s_{n-1}\ldots\s_1)\Delta_n=\Delta_n^2
\end{eqnarray*}
\begin{eqnarray*}
B_n&=&(\s_{n-1}\ldots\s_2\s_{1}^2\s_2\ldots\s_{n-1})\ldots(\s_3\s_2\s_{1}^2\s_2\s_3)(\s_2\s_{1}^2\s_2)\s_{1}^2\\
&=&(\s_{n-1}\ldots\s_2\s_{1})(\s_1\ldots\s_{n-1})\Delta_{n-1}^2\\
&=&(\s_{n-1}\ldots\s_2\s_{1})(\s_1\ldots\s_{n-1})\Delta'_{n-1}\Delta_{n-1}'\\
&=&(\s_{n-1}\ldots\s_2\s_{1})\Delta_{n}'\Delta_{n-1}'=\Delta_{n}(\s_1\ldots\s_{n-1})\Delta_{n-1}'=\Delta_{n}^2
\end{eqnarray*}
and
$$C_n=\Delta'_n.\Delta_n=\Delta_{n}^2.$$
In the product of $n$ brackets of $D_n$ we move $\s_{n-1}$ from the first bracket $n-2$ times (just before the last bracket) and this becomes $\s_1$, then we continue this process with $\s_{n-1}$ of the second  bracket and so on; we obtain
$$
\begin{array}{ll}
D_n& =(\s_1\s_2\ldots\s_{n-1})^n \\
      & = (\s_1\ldots\s_{n-2})(\s_{1}\ldots\s_{n-1})\ldots(\s_{1}\ldots\s_{n-1})\s_1(\s_{1}\ldots\s_{n-1})\\
       &=(\s_1\ldots\s_{n-2})(\s_{1}\ldots\s_{n-2})\ldots(\s_{1}\ldots\s_{n-1})\s_2\s_1(\s_{1}\ldots\s_{n-1})=\,\,\ldots\\
      &=(\s_1\ldots\s_{n-2})(\s_1\ldots\s_{n-2})\ldots(\s_1\ldots\s_{n-2})(\s_{n-1}\ldots\s_1)(\s_1\ldots\s_{n-2})\\
      &=D_{n-1}(\s_{n-1}\ldots\s_1)(\s_1\ldots\s_{n-2})\\
      &=A_{n-1}(\s_{n-1}\ldots\s_1)(\s_1\ldots\s_{n-2})=A_n.
      \end{array}
$$
$D'_n=\Delta_nD_{n}\Delta_{n}^{-1}=D_n.$\\

Using $\a_{ij}=\s_{j-1}\s_{j-2}\ldots\s_{i+1}\s_{i}^2\s_{i+1}^{-1}\ldots\s_{j-1}^{-1}$, we obtain after cancelation
$$\a_{1n}\a_{2n}\ldots\a_{n-1,n}=(\s_{n-1}\s_{n-2}\ldots\s_1)(\s_1\s_2\ldots\s_{n-1})$$
and
\begin{eqnarray*}
E_n&=&(\a_{1n}\a_{2n}\ldots\a_{n-1,n})(\a_{1n-1}\ldots\a_{n-2,n-1})\ldots(\a_{13}\a_{23})\a_{12}\\
&=&(\s_{n-1}\s_{n-2}..\s_{1})(\s_1..\s_{n-1})(\s_{n-2}..\s_{1})(\s_{1}..\s_{n-2})\ldots(\s_2\s_1)(\s_1\s_2)\s_{1}^2\\
&=&B_n,
\end{eqnarray*}
and
$$F_n=F_{n-1}(\a_{1n}\ldots\a_{n-1,n})=B_{n-1}(\s_{n-1}\s_{n-2}\ldots\s_1)(\s_{1}\s_2\ldots\s_{n-1})=B_n\,.\,\,\square$$

\medskip
\end{document}